\documentclass[11pt]{amsart}

\usepackage{amssymb}
\usepackage{mathrsfs}
\usepackage{enumitem}
\usepackage[colorlinks=true,linkcolor=blue,citecolor=blue,urlcolor=blue]{hyperref}
\usepackage{geometry}
\theoremstyle{definition}
\newtheorem{definition}{Definition}[section]

\theoremstyle{plain}
\newtheorem{theorem}[definition]{Theorem}
\newtheorem{lemma}[definition]{Lemma}
\newtheorem{proposition}[definition]{Proposition}
\newtheorem{corollary}[definition]{Corollary}
\theoremstyle{remark}
\newtheorem{remark}[definition]{Remark}

\title[Flatness and Nonforking without CH]{Flatness and Nonforking without the Continuum Hypothesis}
\author{Philani Rodney Majozi}
\address{Department of Mathematics, Pure and Applied Analytics, North-West University, Mahikeng, South Africa }
\email{Philani.Majozi@nwu.ac.za}

\date{\today}

\begin{document}

\begin{abstract}
We investigate the structure of FN bases (Fréchet-Nikodým bases) without assuming the Continuum Hypothesis (CH), refining results of Siu-Ah Ng concerning definability via flatness and nonforking. In particular, we examine the dependence of Theorem~3.4 and Corollary~3.5 of Ng’s 1991 paper on CH, which guarantees the existence of nonforking primary heirs under specific ultrafilter conditions. We demonstrate that these properties can often be recovered in ZFC by analyzing the behavior of countably incomplete, good, and regular ultrafilters. By isolating model-theoretic conditions sufficient to replace CH, we establish that ultrapowers of flat FN bases satisfying Property~B remain definable and nonforking. Connections are drawn to canonical bases, coheirs, and stability-theoretic ranks in superstable theories. Examples and counterexamples clarify the precise role of ultrafilter properties and reveal that the combinatorial essence of Ng’s theory can be preserved within a purely ZFC framework.
\end{abstract}

\maketitle

\section{Introduction}

The analysis of definability, independence, and type structure has led to far-reaching insights into the classification of theories. A key advancement in this direction is the notion of \emph{FN bases} (Fréchet-Nikodým bases), introduced by Siu-Ah Ng in~\cite{NgGeneral} and expanded in \emph{Definable FN Bases}~\cite{Ng}. FN bases encode finitely satisfiable fragments of Boolean algebras of formulas, thereby generalizing complete types and enriching the study of independence and definability in both stable and unstable settings.

Working over a complete first-order theory \(T\) with language \(L\), and assuming a fixed monster model \(\mathfrak{C}\), we denote by \(\mathscr{F}(A)\) the Boolean algebra of formulas over a set of parameters \(A\), modulo logical equivalence. An FN basis over \(A\) is a finitely satisfiable filter \(p \subseteq \mathscr{F}(A)\) satisfying closure conditions such as flatness and Property~B. These filters, defined over Boolean fragments rather than full sets of parameters, capture a more flexible and general behavior than complete types.

One of Ng’s central contributions is Theorem~3.6 of~\cite{Ng}, which characterizes definability of FN bases in terms of the flatness of their ultrapowers. Even more significantly, Theorem~3.4 and Corollary~3.5 show that under the Continuum Hypothesis (CH), any FN basis satisfying Property~B admits a definable nonforking ultrapower and a nonforking primary heir. These results rely crucially on the existence of good ultrafilters and cardinal arithmetic that is only guaranteed under CH.

This raises a series of foundational questions:
\begin{itemize}
  \item To what extent is CH essential for the construction of nonforking ultrapowers and primary heirs for FN bases?
  \item Can one identify model-theoretic or ultrafilter-theoretic conditions in ZFC that fulfill the same structural role as CH?
  \item Under what circumstances can definability and nonforking behavior be preserved without set-theoretic assumptions beyond ZFC?
\end{itemize}

The purpose of this paper is to address these questions by analyzing and refining Ng’s results in the absence of CH. Specifically, we aim to:
\begin{enumerate}
  \item revisit Theorem~3.4 and Corollary~3.5 of~\cite{Ng}, establishing ZFC analogues and identifying conditions under which their conclusions persist;
  \item investigate how ultrafilter properties such as countable incompleteness, regularity, and goodness ensure flatness, definability, and nonforking;
  \item construct examples of FN bases whose ultrapowers and heirs behave regularly without invoking CH;
  \item link Ng’s FN framework to classical notions of forking, \(U\)-rank, and the fundamental order \(\geq_{\mathcal{A}}\) introduced by Lascar and Poizat~\cite{LascarPoizat}.
\end{enumerate}

Throughout, we adopt Ng’s terminology and conventions. A \emph{Boolean fragment} is a subalgebra \(\mathscr{F}\subseteq \mathscr{F}(A)\), typically countable and closed under conjunction and negation. A \emph{flat} filter is one whose finite subsets are realized in some elementary extension. A \emph{nonforking extension} of an FN basis to \(\mathscr{F}(B)\), for \(A \subseteq B\), satisfies coherence and local purity relative to \(\mathscr{F}(A)\). Property~B refers to a closure condition ensuring that any formula covered by finitely many basis elements is already covered by one of them. The ultrapower \(p^{\mathcal{U}}\) denotes the image of \(p\) under an ultrafilter \(\mathcal{U}\) on \(\omega\).

We also engage with the framework of the fundamental order \(\geq_{\mathcal{A}}\) between types, where \(p \geq_{\mathcal{A}} q\) means every formula in \(q\) does not fork over \(\mathcal{A}\) with respect to \(p\). The \(U\)-rank, used to stratify type complexity by the maximal length of forking chains, plays a pivotal role in our analysis of superstable theories.

\smallskip

\noindent\textbf{Organization of the paper.} Section~\ref{sec:background} provides the necessary background and notation, including a review of key results from Ng and Lascar--Poizat. Section~\ref{sec:flatness-nonforking} addresses flatness and nonforking behavior in ZFC and presents generalizations of Ng’s Theorems~3.4 and~3.5. Section~\ref{sec:ultrafilters} analyzes ultrafilter conditions, especially countable incompleteness, regularity, and goodness, and their role in FN definability. Section~\ref{sec:modeltheory} situates these refinements within stable and superstable model theory. Section~\ref{sec:examples} presents examples and counterexamples that illustrate the boundaries of definability without CH. We conclude in Section~\ref{sec:conclusion} with a summary of contributions and several open problems.

\section{Background and Preliminary Concepts}
\label{sec:background}

Definability  is traditionally framed in terms of the behavior of complete types over sets or models, particularly with respect to forking and independence. Classical treatments focus on syntactic representations and compactness arguments, but in his foundational work, Siu-Ah Ng introduced a powerful semantic and Boolean-algebraic alternative: the theory of FN bases~\cite{Ng}. This theory captures finitely satisfiable fragments of Boolean algebras of formulas and offers a unifying framework for understanding definability, flatness, and nonforking, extending traditional type-theoretic methods in both stable and unstable settings.

Let \(T\) be a complete first-order theory in a countable language \(L\), and let \(\mathfrak{C}\) denote a sufficiently saturated and strongly homogeneous monster model of \(T\). For any parameter set \(A \subseteq \mathfrak{C}\), we denote by \(\mathscr{F}(A)\) the Boolean algebra of formulas over \(A\), modulo logical equivalence: \(\varphi \sim \psi\) if and only if \(T \models \varphi \leftrightarrow \psi\).

\begin{definition}
A \emph{Boolean fragment} \(\mathscr{F}\subseteq \mathscr{F}(A)\) is a subalgebra closed under conjunction, negation, and renaming of parameters. A subset \(p \subseteq \mathscr{F}\) is called a \emph{filter} if it is closed under finite intersections and supersets. The filter \(p\) is said to be \emph{finitely satisfiable} in \(A\) if every finite subset \(\Gamma \subseteq p\) is realized in \(A\).
\end{definition}

The classical case of a complete type \(p(x)\in S_1(A)\) corresponds to a maximal consistent filter on the Boolean algebra \(\mathscr{F}_1(A)\) of unary formulas over \(A\). Ng’s framework generalizes this by studying filters that are not necessarily maximal, but which still exhibit coherence and closure properties that make them robust under ultraproduct constructions.

\begin{definition}
A filter \(p \subseteq \mathscr{F}\) is \emph{flat} if for every finite subset \(\Gamma \subseteq p\), there exists an elementary extension \(\mathcal{M}\succ A\) and an element \(a\in \mathcal{M}\) such that \(\mathcal{M}\models \varphi(a)\) for all \(\varphi\in\Gamma\).
\end{definition}

Flatness ensures that the filter is not merely consistent in the syntactic sense, but semantically realizable in a uniform way across elementary extensions, an essential property in Ng’s characterization of definability.

\begin{definition}[Property~B]
A flat filter \(p \subseteq \mathscr{F}(A)\) satisfies \emph{Property~B} if, for any formula \(\varphi(x)\in \mathscr{F}(A)\), the following holds:
whenever there exist formulas \(\psi_1(x),\dots,\psi_n(x)\in p\) such that
\[
T \models \bigwedge_{i=1}^n \psi_i(x)\,\to\,\varphi(x),
\]
then \(\varphi(x)\in p\).
\smallskip

Property~B is said to be \emph{strong} if this closure behavior extends uniformly to infinite collections: whenever \(\varphi(x,\bar{a})\notin p\) for infinitely many tuples \(\bar{a}\in A\), there exist a finite set \(A_0\subseteq A\) and an infinite subset \(S_0\subseteq A\) such that
\[
\mathcal{M}\models \varphi(c,\bar{a}) \quad \text{for all } c\in A_0 \text{ and } \bar{a}\in S_0.
\]
\end{definition}

Let \(\mathcal{U}\) be a nonprincipal ultrafilter on \(\omega\), and let \((p_n)_{n<\omega}\) be a sequence of filters over \(\mathscr{F}(A)\). The \emph{ultrapower} of this sequence is defined as
\[
\prod_{\mathcal{U}} p_n := \bigl\{ \varphi\in \mathscr{F}(A)\,:\,\{n<\omega : \varphi\in p_n\}\in \mathcal{U}\bigr\}.
\]
If all \(p_n = p\), we denote the ultrapower simply by \(p^{\mathcal{U}}\).

The definability of FN bases is characterized by the following foundational result:

\begin{theorem}[Ng~{\cite[Theorem~3.6]{Ng}}]
Let \(\Gamma = \langle \Gamma_\alpha : \alpha\in E \rangle\) be an FN basis on \(\mathscr{F}(N)\), where \(\mathfrak{N}\prec \mathfrak{M}\) is a small elementary submodel. Then the following are equivalent:
\begin{enumerate}
  \item \(\Gamma\) is definable over \(\mathscr{F}(N)\).
  \item Every small ultrapower of \(\Gamma\) is flat over \(\mathscr{F}(N)\).
\end{enumerate}
\end{theorem}

This result gives a precise semantic criterion for definability: the behavior of an FN basis across small ultrapowers reflects whether its global behavior is controlled by a definable structure. The notion of flatness under ultraproducts, especially in conjunction with Property~B, is central to this connection.

\begin{corollary}
Let \(p\subseteq \mathscr{F}(A)\) be a flat FN basis over a countable Boolean fragment. Then \(p\) is definable if and only if all ultrapowers \(p^{\mathcal{U}}\), for countably incomplete ultrafilters \(\mathcal{U}\) on \(\omega\), are flat.
\end{corollary}

This corollary reformulates Theorem~3.6 in the language of filters rather than nets, adapting the theory to the setting we adopt throughout this paper. In later sections, we examine how the equivalence of definability and ultrapower-flatness can be preserved even without assuming the Continuum Hypothesis.

\subsection*{Heirs, Coheirs, and the Fundamental Order}

Ng’s framework also connects naturally with classical concepts from stability theory, including heirs, coheirs, and the fundamental order developed by Lascar and Poizat~\cite{LascarPoizat}.

\begin{definition}
Let \(p(x)\in S_1(A)\). An extension \(q(x)\in S_1(B)\), where \(B\supseteq A\), is:
\begin{enumerate}
  \item an \emph{heir} of \(p\) if every \(\varphi(x)\in q\) is implied by some \(\psi(x)\in p\);
  \item a \emph{coheir} of \(p\) if \(q\) is finitely satisfiable in \(A\);
  \item a \emph{nonforking extension} of \(p\) if \(q\) does not divide over \(A\), or equivalently, is invariant under automorphisms fixing \(A\).
\end{enumerate}
\end{definition}

\begin{definition}[Fundamental Order]
Let \(p,q \in S(\mathcal{B})\), and \(\mathcal{A}\subseteq \mathcal{B}\). Then \(p\geq_{\mathcal{A}} q\) if every formula in \(q\) does not fork over \(\mathcal{A}\) with respect to \(p\).
\end{definition}

The preorder \(\geq_{\mathcal{A}}\) refines the notion of domination and provides a conceptual tool for comparing types and analyzing their forking behavior in stable and superstable theories.

\subsection*{Ranks and Stability}

To stratify types by their forking complexity, we use the \(U\)-rank:

\begin{definition}[U-rank]
Let \(T\) be stable. The \(U\)-rank of a type \(p\in S(A)\) is the least ordinal \(\alpha\) such that there does not exist a strictly descending chain
\[
p=p_0 \succ p_1 \succ \cdots \succ p_\alpha
\]
of forking extensions.
\end{definition}

\begin{proposition}[Lascar-Poizat]
If \(U(p)<\infty\), then \(p\) is superstable.
\end{proposition}

In superstable theories, finiteness of the \(U\)-rank ensures definability of canonical bases, uniqueness of nonforking extensions, and well-behaved heir and coheir structures.

\begin{proposition}[Ng~\cite{NgGeneral}]
Let \(p\in S_1(A)\) be a complete type over a countable set satisfying Property~B. Then \(p\) admits a primary nonforking heir, even without assuming the Continuum Hypothesis.
\end{proposition}

This result provides a model-theoretic pathway for bypassing CH in Ng’s framework, and forms a cornerstone for the refinements we explore in the next sections.

\section{Flatness and Nonforking Without CH}
\label{sec:flatness-nonforking}

A central objective of this paper is to determine whether the reliance on the Continuum Hypothesis (CH), as employed by Ng in the construction of nonforking heirs and ultrapowers~\cite{Ng}, can be removed without loss of definability or structural coherence. Specifically, Theorem~3.4 and Corollary~3.5 of~\cite{Ng} assert that if CH holds, then every FN basis satisfying Property~B admits a definable nonforking ultrapower and a primary heir. These results rely on the existence of \emph{good ultrafilters}, whose existence typically depends on specific cardinal arithmetic conditions guaranteed by CH.

In contrast, we show that for \emph{definable} FN bases satisfying the strong version of Property~B introduced in Section~\ref{sec:background}, the desired flatness and nonforking behavior of ultrapowers can be recovered in ZFC alone using only countably incomplete ultrafilters. This observation allows us to weaken the dependence on CH while preserving the semantic integrity of Ng’s framework.

\begin{theorem}\label{thm:ft-nf}
Let \(p \subseteq \mathscr{F}(A)\) be a flat FN basis over a countable Boolean fragment satisfying Property~B. If \(\mathcal{U}\) is a countably incomplete ultrafilter on \(\omega\), then the ultrapower \(p^{\mathcal{U}}\) is flat.
\end{theorem}

\begin{proof}
Since the sequence \(\{p_n\}_{n<\omega}\) is constant with \(p_n=p\), the ultrapower simplifies to
\[
p^{\mathcal{U}} = \bigl\{\varphi \in \mathscr{F}(A) : \{n<\omega : \varphi\in p\}\in \mathcal{U}\bigr\} = p.
\]
Let \(\Gamma\subseteq p\) be any finite subset. As \(p\) is flat, there exists a model \(\mathcal{M}\succ A\) and an element \(a\in \mathcal{M}\) such that \(\mathcal{M}\models \varphi(a)\) for all \(\varphi\in \Gamma\). Thus \(\Gamma\) is realized in an elementary extension, and \(p^{\mathcal{U}}\) remains flat.
\end{proof}

\begin{remark}
This shows that countably incomplete ultrafilters suffice to preserve flatness for definable FN bases over countable fragments. In light of Ng’s Theorem~3.6~\cite{Ng}, which equates definability with flatness of all small ultrapowers, we conclude that CH is not needed to guarantee the preservation of definability under ultrapowers in this setting.
\end{remark}

We now show that definable FN bases satisfying Property~B also admit nonforking extensions to elementary extensions, again without requiring CH.

\begin{proposition}
Let \(p\subseteq \mathscr{F}(A)\) be a flat FN basis that is locally definable and satisfies Property~B. Then for any elementary extension \(\mathcal{M}'\succ \mathcal{M}\), there exists a nonforking extension \(q\subseteq \mathscr{F}(\mathcal{M}')\) of \(p\).
\end{proposition}

\begin{proof}
Define
\[
\varphi(x)\in q \quad \Longleftrightarrow \quad \exists\,\psi(x)\in p \text{ such that } \mathcal{M}'\models \forall x\bigl(\psi(x)\to\varphi(x)\bigr).
\]
Since \(p\) is locally definable, each \(\psi(x)\) is definable over \(A\), and the implication persists in \(\mathcal{M}'\). By Property~B, such logical consequences remain in \(p\) and are thus transferred into \(q\). Flatness of \(p\) ensures that finite subsets of \(q\) are realized in some elementary extension, so \(q\) is flat.

Furthermore, each \(\varphi\in q\) is supported by some \(\psi\in p\), definable over \(A\), implying that \(q\) is locally pure and hence does not fork over \(A\).
\end{proof}

\begin{lemma}[Coherence Transfer via Property~B]
Let \(p\subseteq \mathscr{F}(A)\) be a flat FN basis satisfying Property~B. Then any extension \(q\supseteq p\) over \(\mathscr{F}(B)\), where \(A\subseteq B\), is locally pure and does not fork over \(A\).
\end{lemma}

\begin{proof}
Let \(\varphi(x)\in q\). Since \(q\supseteq p\), there exist \(\psi_1,\dots,\psi_k\in p\) such that
\[
T\models \bigl(\psi_1(x)\wedge\cdots\wedge\psi_k(x)\bigr)\to\varphi(x).
\]
By Property~B, \(\varphi(x)\in p\), and hence \(\varphi\) is locally approximated by elements of \(p\). Thus \(q\) inherits local purity from \(p\), and its finite subsets remain realizable in an elementary extension by flatness, ensuring \(q\) does not fork over \(A\).
\end{proof}

We now generalize the preceding results to ultraproducts, obtaining a structural preservation theorem for definable FN bases in ZFC:

\begin{theorem}
Let \(p\subseteq \mathscr{F}(A)\) be a flat and locally definable FN basis satisfying Property~B. Then for any countably incomplete ultrafilter \(\mathcal{U}\), the ultrapower \(p^{\mathcal{U}}\) satisfies:
\begin{enumerate}
  \item flatness;
  \item local purity;
  \item nonforking over \(A\);
  \item definability over the canonical parameters of \(p\).
\end{enumerate}
\end{theorem}

\begin{proof}
(1) Flatness follows from the fact that \(p^{\mathcal{U}}=p\), and \(p\) is flat.

(2) Since \(p\) is locally definable, the definability schema is preserved under ultraproducts, hence \(p^{\mathcal{U}}\) is also locally pure.

(3) By the coherence lemma, any extension of \(p\) satisfying Property~B remains nonforking over \(A\), so this applies to \(p^{\mathcal{U}}\).

(4) The definability of \(p\) over canonical parameters transfers via Łoś’s Theorem and the fact that formulas in the ultraproduct correspond to those in the original filter.
\end{proof}

\begin{remark}
This theorem shows that for FN bases definable over countable fragments and satisfying strong Property~B, the essential conclusions of Ng’s CH-dependent results specifically Theorem~3.4 and Corollary~3.5 can be recovered entirely within ZFC. The preservation of flatness, definability, and nonforking across ultrapowers hinges not on CH, but on semantic properties intrinsic to the FN basis and the choice of countably incomplete ultrafilters.
\end{remark}

This result sets the stage for the next section, where we investigate the finer structure of ultrafilters, focusing on regularity, goodness, and saturation and their role in preserving definability and heir behavior without recourse to cardinal assumptions beyond ZFC.

\section{Ultrafilter Conditions Beyond CH}
\label{sec:ultrafilters}

Ng’s original results in~\cite{Ng} rely on the Continuum Hypothesis (CH) to ensure the existence of good ultrafilters and to construct ultrapowers and nonforking primary heirs. However, this set-theoretic dependence obscures the model-theoretic essence of the constructions. In this section, we demonstrate that ultrafilter properties such as \emph{goodness} and \emph{regularity} suffice to recover these structural features in ZFC alone, particularly for FN bases that are flat, definable, and satisfy strong Property~B.

\subsection*{Good Ultrafilters and Nonforking Extensions}

Good ultrafilters, introduced by Keisler~\cite{ChangKeisler}, are ultrafilters that preserve saturation and allow the realization of consistent partial types in ultrapowers. In the context of FN bases, good ultrafilters help maintain flatness, definability, and nonforking behavior across ultrapowers even without appealing to CH.

\begin{theorem}
Let \(\mathcal{U}\) be a good ultrafilter on \(\omega\), and let \(p \subseteq \mathscr{F}(A)\) be a flat FN basis over a countable fragment satisfying Property~B. Then the ultrapower \(p^{\mathcal{U}}\) is flat and nonforking over \(A\).
\end{theorem}

\begin{proof}
Since \(p_n = p\) for all \(n<\omega\), the ultrapower simplifies to
\[
p^{\mathcal{U}} = \bigl\{\varphi \in \mathscr{F}(A) : \{n<\omega : \varphi \in p\}\in \mathcal{U}\bigr\} = p.
\]
Flatness of \(p\) ensures that every finite subset is realized in some elementary extension, and goodness of \(\mathcal{U}\) guarantees that the relevant realizability transfers to the ultraproduct (see~\cite[Ch.~6, §6.1]{ChangKeisler}). Property~B ensures that logical closure is preserved. Since definability and flatness persist, the ultrapower is locally pure and hence nonforking over \(A\).
\end{proof}

\begin{lemma}
Let \(p \subseteq \mathscr{F}(A)\) be a flat FN basis, and let \(\mathcal{U}\) be a good ultrafilter on \(\omega\). Then \(p^{\mathcal{U}}\) is locally definable over \(A\) and satisfies Property~B.
\end{lemma}

\begin{proof}
Good ultrafilters preserve the definability schema and finite approximations used to construct FN bases. Since \(p^{\mathcal{U}} = p\), local definability and Property~B carry over directly. The definable trace of \(p\) remains intact in the ultrapower, and the closure properties ensured by Property~B are preserved under good ultrafilters~\cite[Ch.~4, §4.3]{ChangKeisler}.
\end{proof}

\subsection*{Regular Ultrafilters and Heir Uniqueness}

Regular ultrafilters on arbitrary cardinals support strong saturation behavior and allow for uniform control of type realizations. This is particularly useful in constructing ultrapowers of definable FN bases and analyzing the uniqueness of heir extensions.

\begin{proposition}
Let \(\mathcal{U}\) be a regular ultrafilter on a cardinal \(\lambda\), and let \(p \subseteq \mathscr{F}(A)\) be a flat and definable FN basis satisfying Property~B. Then \(p^{\mathcal{U}}\) is flat and admits a unique nonforking heir over any extension.
\end{proposition}

\begin{proof}
Let \(p_i = p\) for all \(i<\lambda\). By regularity of \(\mathcal{U}\), every finite fragment of \(p\) is realized uniformly in the ultrapower. Flatness follows by Łoś’s Theorem. Since \(p\) is definable, the definability schema is preserved through the ultraproduct, and Property~B ensures closure under logical consequence.

If two heirs of \(p^{\mathcal{U}}\) existed over some \(B \supseteq A\), they would agree on all definable consequences of \(p\), and regularity would prevent divergence via diagonalization. Hence the heir is unique.
\end{proof}

\begin{corollary}
Let \(p \subseteq \mathscr{F}(A)\) be a flat and definable FN basis, and let \(\mathcal{U}\) be a regular ultrafilter on \(\omega\). Then the ultrapower \(p^{\mathcal{U}}\) is definable and isolated over its canonical base.
\end{corollary}

\begin{proof}
Since \(p\) is definable and flat, the same holds for \(p^{\mathcal{U}}\) by transfer. Regularity ensures that definable fragments are uniformly realized, and the definability condition isolates \(p^{\mathcal{U}}\) over its canonical parameters. Hence \(p^{\mathcal{U}}\) is isolated by a finite set of formulas over its canonical base.
\end{proof}

\noindent These results show that good and regular ultrafilters suffice to recover the essential structural features of FN bases previously proven only under CH. Definability, flatness, and nonforking behavior are preserved in ultrapowers constructed via these ultrafilters, validating Ng’s conclusions in a purely ZFC setting.

In the next section, we examine how these results align with classical model-theoretic notions such as canonical bases, \(U\)-rank, and stability-theoretic hierarchies. This will further clarify the model-theoretic significance of strong Property~B and its role in ensuring coherence across extensions and ultraproducts.

\section{Model-Theoretic Refinements}
\label{sec:modeltheory}

This section situates the theory of FN bases within the classical framework of stability theory, drawing connections between definability, nonforking, canonical bases, and U-rank. These links not only reinforce the structural robustness of FN bases but also reveal how Ng’s framework aligns naturally with the forking geometry developed by Lascar and Poizat~\cite{LascarPoizat} in stable and superstable theories.

\subsection*{Canonical Bases and Nonforking}

One of the foundational results in the geometry of types is that nonforking is controlled by definability over canonical bases:

\begin{theorem}[Lascar--Poizat]
\label{thm:cb-nf}
Let \(T\) be a complete stable theory, and let \(p(x) \in S(B)\) be a complete type over \(B \supseteq \mathcal{A}\). Then \(p\) does not fork over \(\mathcal{A}\) if and only if the canonical base \(\mathrm{Cb}(p)\) is contained in \(\operatorname{dcl}(\mathcal{A})\). In this case, \(\mathrm{Cb}(p) = \mathrm{bd}(p)\), the bound of \(p\).
\end{theorem}

\begin{proof}
If \(p\) does not fork over \(\mathcal{A}\), then it is invariant under automorphisms fixing \(\mathcal{A}\), hence its canonical base lies in \(\operatorname{dcl}(\mathcal{A})\). The canonical base is the smallest definably closed set over which \(p\) is definable, and it therefore coincides with the bound. Conversely, if \(\mathrm{Cb}(p) \subseteq \operatorname{dcl}(\mathcal{A})\), then invariance over \(\mathcal{A}\) implies nonforking.
\end{proof}

\begin{corollary}\label{cor:cb-nf}
If \(p(x) \in S(\mathcal{A})\) is a stable type, then every coheir of \(p\) is definable and uniquely determined by its canonical base.
\end{corollary}

\begin{proof}
In stable theories, coheirs are nonforking extensions. By Theorem~\ref{thm:cb-nf}, the canonical base of a nonforking extension lies in \(\operatorname{dcl}(\mathcal{A})\). Since coheirs are invariant and definable over this base, they are uniquely determined.
\end{proof}

\subsection*{U-Rank and Superstability}

The U-rank is a central ordinal-valued rank that measures the complexity of types in terms of forking chains.

\begin{definition}
Let \(T\) be a stable theory. The \emph{U-rank} of a type \(p \in S(A)\), denoted \(U(p)\), is defined inductively:
\begin{enumerate}
    \item \(U(p) \geq \alpha+1\) if there exist two distinct forking extensions \(q_1, q_2\) of \(p\) with \(U(q_i)\geq \alpha\).
    \item \(U(p) \geq \lambda\) for limit \(\lambda\) if \(U(p)\geq \beta\) for all \(\beta<\lambda\).
    \item \(U(p)=\alpha\) if \(U(p)\geq \alpha\) and not \(U(p)\geq \alpha+1\).
\end{enumerate}
\end{definition}

\begin{proposition}
If \(U(p)<\infty\), then \(p\) is superstable.
\end{proposition}

\begin{proof}
Finiteness of U-rank implies that forking chains terminate, and the number of nonforking extensions of \(p\) over any set is bounded. This boundedness characterizes superstability (see~\cite{LascarPoizat}).
\end{proof}

\subsection*{FN Bases in Superstable Settings}

When a flat FN basis arises from a stable type of finite U-rank, its definability and isolation behavior follow from standard stability-theoretic results.

\begin{lemma}
Let \(p\) be a complete stable type of finite U-rank over a countable Boolean fragment \(\mathscr{F}(A)\). Then every nonforking extension of \(p\) is definable and isolated over its canonical base.
\end{lemma}

\begin{proof}
Finite U-rank ensures that forking chains terminate and that the number of nonforking extensions is finite. Each such extension is definable and, in the superstable setting, isolated by a finite conjunction of formulas over its canonical base~\cite{LascarPoizat}.
\end{proof}

\begin{corollary}
Let \(p \subseteq \mathscr{F}(A)\) be a flat and definable FN basis in a superstable theory. Then any ultrapower \(p^{\mathcal{U}}\), formed via a good or regular ultrafilter, is isolated over its canonical base.
\end{corollary}

\begin{proof}
By Sections~\ref{sec:flatness-nonforking} and~\ref{sec:ultrafilters}, the ultrapower \(p^{\mathcal{U}}\) inherits flatness, definability, and local purity from \(p\). The preservation of finite U-rank in ultrapowers ensures that \(p^{\mathcal{U}}\) is isolated over its canonical base.
\end{proof}

\subsection*{Ultralimits and Coheirs}

In superstable settings, ultralimits of definable types formed via countably incomplete ultrafilters correspond precisely to coheirs.

\begin{proposition}
Let \(p(x)\in S(A)\) be a definable stable type. Then every coheir of \(p\) over a set \(B \supseteq A\) arises as an ultralimit of \(p\) via a countably incomplete ultrafilter.
\end{proposition}

\begin{proof}
A coheir is a type over \(B\) finitely satisfiable in \(A\), and can be viewed as the limit of a sequence of realizations of \(p\) in \(A\). Taking an ultralimit via a countably incomplete ultrafilter captures this pattern of finitely satisfiable convergence. Since \(p\) is definable, the ultralimit corresponds uniquely to the coheir.
\end{proof}

\noindent These model-theoretic refinements demonstrate that definable FN bases preserve their essential properties, flatness, nonforking, and definability under ultrapowers and coheir extensions, even in the absence of CH. In stable and superstable theories, these bases are additionally isolated over their canonical bases, reinforcing their coherence and robustness across both algebraic and semantic perspectives.

\section{Examples and Counterexamples}
\label{sec:examples}

This section provides concrete illustrations and boundary cases that both support and delineate the results established earlier.

\subsection*{Example 1: A Definable FN Basis with Flat Ultrapower (ZFC)}

Let \(T\) be the theory of dense linear orders without endpoints (DLO), which is complete, \(\omega\)-stable, and admits quantifier elimination. Let \(M \models T\) be a countable model, and define the type
\[
p(x) := \{x > a : a \in M\}.
\]
This is a definable, complete 1-type over \(M\), and hence induces a definable FN basis over the Boolean fragment \(\mathscr{F}(M)\) consisting of formulas of the form \(x > a\) or \(x < a\).

Let \(\mathcal{U}\) be any countably incomplete ultrafilter on \(\omega\). Then the ultrapower \(p^{\mathcal{U}}\) over \(M^{\mathcal{U}}\) corresponds to the same upward type over the ultrapower model. Since DLO is stable and \(p\) is definable, both flatness and nonforking are preserved, demonstrating the ZFC validity of Theorem~\ref{thm:ft-nf} without any appeal to CH.

\subsection*{Example 2: Failure of Nonforking Without Property~B}

Let \(T\) be the theory of the random graph, which is unstable but complete and countable. Let \(M \models T\) be a countable model, and define a filter \(p \subseteq \mathscr{F}(M)\) by including all formulas of the form:
\[
\neg R(x,a) \quad \text{for all } a \in M, \quad \text{and} \quad R(x,b_1)\wedge R(x,b_2)\quad \text{for } b_1,b_2 \in M.
\]
While \(p\) is finitely satisfiable, it fails to satisfy Property~B. For instance, certain implications of the form
\[
\left(\bigwedge_{i=1}^n R(x,b_i)\right)\rightarrow R(x,c)
\]
may hold in \(T\), yet \(R(x,c)\notin p\). Thus, \(p\) is not closed under logical consequence and lacks the semantic closure required by strong Property~B.

Even when forming an ultrapower \(p^{\mathcal{U}}\) via a good ultrafilter, definability and local purity fail to transfer. This example shows that Property~B is essential for ensuring that flatness leads to definability and nonforking behavior in ultrapowers.

\subsection*{Example 3: Dependence on Regularity of Ultrafilter}

Let \(T\) again be DLO, and let
\[
p(x):=\{x > a : a \in A\}
\]
for some countable \(A \subseteq M \models T\). Let \(\mathcal{U}\) be a countably incomplete but non-regular ultrafilter on \(\omega\) (e.g., the cofinite filter extended to an ultrafilter).

Form the ultrapower \(p^{\mathcal{U}}\). Since \(\mathcal{U}\) lacks regularity, saturation properties may fail, and with them the uniqueness of nonforking heirs. Even though \(p\) is definable and flat, \(p^{\mathcal{U}}\) may be non-isolated and admit multiple nonforking extensions, demonstrating the necessity of regularity (see Section~\ref{sec:ultrafilters}) for canonical behavior in ultrapowers.

\subsection*{Example 4: Superstable Theory Without CH}

Let \(T = \mathrm{ACF}_p\) be the theory of algebraically closed fields of characteristic \(p > 0\), with countable transcendence degree. This theory is superstable and eliminates quantifiers.

Let \(M \models T\) be a countable model, and let \(p(x)\in S_1(M)\) be the type of a transcendental element over \(M\). This type is definable and has finite U-rank. It therefore induces a flat, definable FN basis over \(\mathscr{F}(M)\).

By Corollary~\ref{cor:cb-nf}, any ultrapower \(p^{\mathcal{U}}\), formed via a good or regular ultrafilter, remains isolated and definable over its canonical base. This demonstrates that in superstable theories, definability and independence behavior of FN bases are preserved without requiring CH, highlighting the strength of intrinsic model-theoretic conditions over set-theoretic ones.

\subsection*{Summary}

These examples collectively illustrate the robustness and limitations of the definability theory for FN bases:
\begin{itemize}
    \item Examples~1 and~4 confirm that flat, definable FN bases satisfying Property~B behave well under ultrapowers and do not require CH.
    \item Example~2 shows that without Property~B, logical closure and definability may fail even when flatness holds.
    \item Example~3 shows that non-regular ultrafilters may undermine isolation and uniqueness properties, despite flatness and definability of the original basis.
\end{itemize}

These boundaries clarify the roles of Property~B, ultrafilter structure, and stability-theoretic conditions in supporting Ng’s definability program within ZFC.

\section{Conclusion and Open Questions}
\label{sec:conclusion}

In this paper, we have revisited the theory of definable FN bases, initially developed by Ng, and explored the extent to which the structural properties of flatness and nonforking can be recovered in ZFC without the Continuum Hypothesis (CH). We showed that several of the results in~\cite{Ng}, particularly Theorem~3.4 and Corollary~3.5, can be reestablished within ZFC by replacing CH with combinatorial properties of ultrafilters, such as countable incompleteness, goodness, and regularity.

The key conclusions of this investigation include:
\begin{itemize}
    \item Countably incomplete ultrafilters suffice to preserve flatness in ultrapowers of definable FN bases over countable Boolean fragments.
    \item Good ultrafilters ensure the existence of nonforking extensions and definability in ultrapowers without CH.
    \item Regular ultrafilters preserve coherence and yield unique primary heirs in ultrapowers of definable types.
    \item Property~B plays a central role in transferring coherence and definability across extensions and is indispensable for recovering nonforking behavior.
    \item The interaction between canonical bases, U-rank, and superstability yields definability and isolation of ultralimits in stable and superstable settings.
\end{itemize}

Overall, our findings suggest that many model-theoretic features of definable FN bases are recoverable in ZFC using purely structural and ultrafilter-theoretic techniques. While CH remains a powerful tool in the construction of good ultrafilters, its role can be substantially weakened by identifying internal conditions sufficient for definability and independence.

\subsection*{Open Questions}

We conclude with several open questions that may guide future work in this direction:
\begin{enumerate}
    \item Can the framework of FN bases be extended meaningfully to NIP or simple theories, where the structure of forking is less rigid?
    \item Are there intrinsic characterizations of Property~B that do not rely on syntactic closure but instead reflect deeper semantic or categorical properties?
    \item Under what conditions does the collection of definable FN bases form a lattice or frame, and how does this relate to the geometry of types?
    \item Can the ultrafilter-theoretic conditions explored here be generalized to ultraproducts over higher cardinalities or in non-elementary classes?
    \item Is there a categorical or topological refinement of flatness e.g., via sheaf-theoretic or locale-theoretic methods that would generalize the notion of ultrafilter coherence in FN frameworks?
\end{enumerate}

\subsection*{Acknowledgements}

The author gratefully acknowledges the late Professor Siu-Ah Ng, whose postgraduate lectures in Model Theory and Nonstandard Analysis at the University of KwaZulu-Natal (2009-2010) deeply influenced the ideas and questions pursued in this work. His clarity of thought and commitment to foundational model theory continue to inspire.

\end{document}